\documentclass[11pt,a4paper]{article}
\usepackage[top=3cm, bottom=3cm, left=3cm, right=3cm]{geometry}
\usepackage[latin1]{inputenc}
\usepackage{csquotes}
\usepackage{amsmath}
\usepackage{amsthm}
\usepackage{amsfonts}
\usepackage{amssymb}
\usepackage{graphics}
\usepackage{float}
\usepackage[hang,flushmargin]{footmisc} 

\usepackage{amsmath,amsthm,amssymb,graphicx, multicol, array}
\usepackage{enumerate}
\usepackage{enumitem}
\usepackage{xcolor}
\usepackage{currfile}
\usepackage{mathrsfs}
\usepackage{multicol}
\usepackage[pagebackref,bookmarks,colorlinks,breaklinks,linktoc=page]{hyperref}
\hypersetup{linkcolor=blue,citecolor=blue,filecolor=blue,urlcolor=blue} 

\usepackage[hang,flushmargin]{footmisc} 
\usepackage{tabto}

\usepackage{titletoc}

\titlecontents{section}
[0pt]                                               
{}%
{\contentsmargin{0pt}                               
	\thecontentslabel\enspace%
	\normalsize}
{\contentsmargin{0pt}\normalsize}                        
{\titlerule*[.5pc]{.}\contentspage}                 
[]      

\DeclareMathOperator{\li}{li}
\DeclareMathOperator{\ord}{ord}

\DeclareMathOperator{\lcm}{lcm}

\newtheorem{thm}{Theorem}[section]
\newtheorem{theorem}{Theorem}[section]
\newtheorem{lem}{Lemma}[section]
\newtheorem{lemma}{Lemma}[section]
\newtheorem{conj}{Conjecture}[section]

\newtheorem{dfn}{Definition}[section]
\newtheorem{exe}{Exercise}[section]

\newcommand{\N}{\mathbb{N}}
\newcommand{\Z}{\mathbb{Z}}
\newcommand{\Q}{\mathbb{Q}}

\newcommand{\C}{\mathbb{C}}
\newcommand{\F}{\mathbb{F}}

\usepackage{fancyhdr}
\usepackage{titletoc}

\titlecontents{section}
[0pt]                                               
{}%
{\contentsmargin{0pt}                               
	\thecontentslabel\enspace%
	\normalsize}
{\contentsmargin{0pt}\normalsize}                        
{\titlerule*[.5pc]{.}\contentspage}                 
[] 
\title{Prescribed Primitive Roots and the Least Primes}
\date{}
\author{N. A. Carella}

\begin{document}
\maketitle
\thispagestyle{empty}
\date{}

\begin{abstract}
Let $q\ne \pm1,v^2$ be a fixed integer. The least prime number $p \geq3 $ such that $q$ is a primitive root modulo $p$ is conjectured to be $p\ll (\log q)(\log \log q)),$ where $\gcd(p,q)=1$. This note proves the existence of a small subset of primes $p$ that satisfies the stated upper bound. 
\let\thefootnote\relax\footnote{ \today \date{} 
	}
\end{abstract}

\tableofcontents

\section{Introduction}\label{S3232.000-I}\hypertarget{S3232.000-I}
It is a routine calculations, using the Laws of Quadratic Reciprocity, to verify that the subset of integers $\mathcal{Q}_0=\{q=5a+2\text{ or } q=5a+3:a\geq0\}$ are primitive roots modulo $p=5$. Here, $p=5$ is the least such prime. However, it is probably difficult or impossible to use the same analysis to verify which nonsquare integers in the subset of integers $\mathcal{Q}_1=\{q=5a+1\text{ or } q=5a+4:a\geq1\}$ are primitive roots of small primes $p\geq3$, which are bounded by a constant $p\leq c=c(q)$.

\begin{conj}   \label{conj3232.000D}\hypertarget{conj3232.000D}  {\normalfont ({\color{blue}Granville, \cite{MO2021}})} Let $q\ne 1, n^2$ be a fixed integer. Then, there exists a prime $p\ll (\log q)(\log \log q)^3$ such that $q$ is a primitive root modulo $p$.
\end{conj} 

In addition, the conjecture claims that for any prescribed nonsquare integer $q\geq2$, a small finite subset of Germain primes
\begin{equation} \label{eq3232.310}
\mathcal{G}=\{p=2r+1: r \text{ and } p \text{ are primes}\}=\{3, 5, 7,23,43,\ldots\}
\end{equation}
contains a least prime $p$ such that $q$ is a primitive root modulo $p$, with very  few exceptions, see \cite{MO2021}. The subset of Germain primes is suitable for this application because each of these primes has a very simple primitive root test. A version of this conjecture is considerable in this note. 

\begin{theorem}   \label{thm3232.100}\hypertarget{thm3232.100}  If $q\ne 1, r^2$ is a fixed integer, then  exists a prime $p\ll (\log q)(\log\log q)$, unconditionally.
\end{theorem} 

The preliminary notation, definitions, background results are discussed and proved in \hyperlink{S3242.100-RT}{Section} \ref{S3242.100-RT} to \hyperlink{S3282.500-E}{Section} \ref{S3282.500-E}. A discussion of the heuristic appears in in \hyperlink{S3292.333-H}{Section} \ref{S3292.333-H}. The proof of \hyperlink{thm3232.100}{Theorem} \ref{thm3232.100} is presented in \hyperlink{S3292.900-T}{Section} \ref{S3292.900-T}.

\section{Primitive Roots Tests}  \label{S3242.100-RT}\hypertarget{S3242.100-RT}
For any prime $p \geq 3$, the multiplicative group $G$ of the prime finite fields $\mathbb{F}_p$ is a cyclic group of cardinality $p-1=\#G$. Similar result is true for any finite extension $\mathbb{F}_q$ of $\mathbb{F}_p$, where $q=p^k$ is a prime power. \\

\begin{dfn} \label{dfn3242.100D}\hypertarget{dfn3232.100D} {\normalfont The \textit{multiplicative order} of an element $u\ne0$ in the cyclic group $\mathbb{F}_p^\times$ is defined by $\ord_p(u)=\min \{k \in \mathbb{N}: u^k \equiv 1 \bmod p \}$. An element is a \textit{primitive root} if and only if $\ord_p(u)=p-1$. 
}
\end{dfn}

The Euler totient function counts the number of relatively prime integers \(\varphi (n)=\#\{ k:\gcd (k,n)=1 \}\). This counting function is compactly expressed by
the analytic formula \(\varphi (n)=n\prod_{p \mid n}(1-1/p),n\in \mathbb{N} .\)\\

\begin{lem} {\normalfont (Fermat-Euler)} \label{lem3242.200}\hypertarget{lem3242.200}If \(a\in \mathbb{Z}\) is an integer such that \(\gcd (a,n)=1,\) then \(a^{\varphi (n)}\equiv
	1 \bmod n\).
\end{lem}

\begin{lem} \label{lem3242.500}\hypertarget{lem3242.500}  {\normalfont (Primitive root test in $\mathbb{F}_p$)} An integer $u \in \Z$ is a primitive root modulo an integer $n \in \N$ if and only if 
\begin{equation}\label{eq3242.500}
u^{\varphi(n)/p} -1\not \equiv 0 \mod  n
\end{equation}
for all prime divisors $p \mid \varphi(n)$.
\end{lem}
The primitive root test is a special case of the Lucas primality test, introduced in {\color{blue}\cite[p.\ 302]{ LE1878}}. A more recent version appears in {\color{blue}\cite[Theorem 4.1.1]{CP2005}}, and similar sources. \\

The Carmichael function is basically a refinement of the Euler totient function to the finite ring \(\mathbb{Z}/n \mathbb{Z}\). 

\begin{dfn} Given an integer
\(n=p_1^{v_1}p_2^{v_2}\cdots  p_t^{v_t}\), the Carmichael function is defined by
\begin{equation}
	\lambda (n)=\lcm \left (\lambda \left(p_1^{v_1}\right),\lambda \left (p_2^{v_2}\right ) \cdots  \lambda \left (p_t^{v_t}\right ) \right )
	=\prod _{p^v  \mid \mid  \lambda (n)} p^v,
\end{equation}
where the symbol \(p^v \mid \mid n,\nu \geq 0\), denotes the maximal prime power divisor of \(n\geq 1\), and 
\begin{equation}
	\lambda
	\left(p^v\right)= \left \{
	\begin{array}{ll}
		\begin{array}{ll}
			\varphi \left(p^v\right) 
			& \text{ if } p\geq 3\text{ or } v\leq 2,  \\
			2^{v-2} 
			& \text{ if } p=2 \text{ and }v \geq 3. \\
		\end{array}
		
	\end{array} \right . 
\end{equation}
\end{dfn} 
The two functions coincide, that is, \(\varphi(n)=\lambda (n)\) if \(n=2,4,p^m,\text{ or } 2p^m,m\geq 1\). And \(\varphi \left(2^m\right)=2\lambda
\left(2^m\right)\). In a few other cases, there are some simple relationships between \(\varphi (n) \text{ and } \lambda (n)\). In fact, it seamlessly
improves the Fermat-Euler Theorem: The improvement provides the least exponent \(\lambda (n) \mid  \varphi (n)\) such that \(a^{\lambda (n)}\equiv
1 \bmod n\). The ratio $\varphi (n)/\lambda (n) \geq1$ has many interesting properties studied in the literature.\\

\begin{lem} \label{lem3242.550}\hypertarget{lem3242.550} { \normalfont (\cite{CR1910})}  Let \(n\in \mathbb{N}\) be any given integer. Then
\begin{enumerate} [font=\normalfont, label=(\roman*)]
\item The congruence \(a^{\lambda (n)}\equiv
1 \bmod n\) is satisfied by every integer \(a\geq 1\) relatively prime to \(n\), that is \(\gcd (a,n)=1\).

\item In every congruence \(x^{\lambda (n)}\equiv 1 \bmod n\), a solution \(x=u\) exists which is a primitive root \(\bmod  n\), and for any such
solution \(u\), there are \(\varphi (\lambda (n))\) primitive roots congruent to powers of \(u\).
\end{enumerate} 
\end{lem}

\begin{proof}[\textbf{Proof}] (i) The number \(\lambda (n)\) is a multiple of every \(\lambda \left(p^v\right)=\varphi\left(p^v\right) \) such that \(p^v \mid  n\). Ergo, for any relatively prime integer \(a\geq 2\), the system of congruences 
\begin{equation}
a^{\lambda (n)}\equiv 1\bmod p_1^{v_1}, \quad a^{\lambda (n)}\equiv 1\bmod p_2^{v_2}, \quad \ldots, \quad a^{\lambda (n)}\equiv 1\bmod p_t^{v_t},
\end{equation}
where \(t=\omega (n)\) is the number of prime divisors in \(n\), is valid. 
\end{proof}

\begin{dfn}
	An integer \(u\in \mathbb{Z}\) is called a \textit{$\lambda$-primitive root} \(\text{mod } n \) if the least exponent \(\min  \left\{ m\in \mathbb{N}:u^m\equiv 1 \bmod n \right\}=\lambda
	(n)\).
\end{dfn}

\begin{lem} \label{lem3242.600}\hypertarget{lem3242.600}  {\normalfont (Primitive root test in $\mathbb{Z}/n\mathbb{Z}$)} An integer $u \in \Z$ is a primitive root modulo an integer $n \in \N$ if and only if 
\begin{equation}\label{eq3242.600}
u^{\lambda(n)/p} -1\not \equiv 0 \mod  n
\end{equation}
for all prime divisors $p \mid \lambda(n)$.
\end{lem}

\begin{lem} \label{lem3242.700}\hypertarget{lem3242.700}  {\normalfont (Primitive root lift)} Let $n$, and $u \in \mathbb{N}$ be integers, $\gcd(u,n)=1$. If $u$ is a primitive root modulo $p^k$ for each prime power divisor $p^k \mid n$, then, the integer $u \ne \pm 1, v^2$ is a primitive root modulo $n$.
\end{lem}

\begin{proof}[\textbf{Proof}] Without loss in generality, let $n=pq$ with $p\geq 2$ and $q\geq 2$ primes. Let $u$ be a primitive root modulo $p$ and modulo $q$ respectively. Then
\begin{equation} \label{eq3242.700}
u^{(p-1)/r}  -1\not \equiv 0 \bmod  p  \qquad  \text{  and  } \qquad u^{(q-1)/s} -1 \not \equiv 0 \bmod  q,
\end{equation}
for every prime $r \mid p-1$, and every prime $s \mid q-1$ respectively, see \hyperlink{S3242.500}{Lemma} \ref{lem3242.500}. Now, suppose that $u$ is not a primitive root modulo $n$. In particular,
\begin{equation}\label{eq3242.720}
u^{\lambda(n)/t} -1\equiv 0 \bmod  n  
\end{equation}
for some prime divisor $t \mid \lambda(n)$.\\

Let $v_t(\lambda(n))$, $v_t(p-1)$, and $v_t(q-1)$ be the $t$-adic valuations of these integers. Since $\lambda(n)=\lcm(\varphi(p-1), \varphi(q-1))$, it follows that at least one of the relations
\begin{equation}\label{eq3242.730}
v_t(\lambda(n))=v_t(p-1)  \qquad  \text{ or } \qquad v_t(\lambda(n))=v_t(q-1) 
\end{equation}
is valid. As consequence, at least one of the congruence equations
\begin{equation}\label{eq3242.740}
u^{\lambda(n)/t} -1\equiv 0 \bmod  n  \qquad  \Longleftrightarrow \qquad u^{\lambda(n)/t}-1 \equiv 0\mod p
\end{equation}
or
\begin{equation}\label{eq3242.750}
u^{\lambda(n)/t} -1\equiv 0 \bmod  n  \qquad  \Longleftrightarrow \qquad u^{\lambda(n)/t}-1 \equiv0 \mod q
\end{equation}
fails. But, this in turns, contradicts the relations in (\ref{eq3242.700}) that $u$ is a primitive root modulo both $p$ and $q$. Therefore, $u$ is a primitive root modulo $n$. 
\end{proof}

\begin{lem} \label{lem3242.850}\hypertarget{lem3232.850} The integer $2$ is a quadratic residue, {\normalfont (}quadratic nonresidue{\normalfont )} of the primes of the form $p=8k\pm1$, {\normalfont (}$p=8k\pm3$ respectively{\normalfont )}. Equivalently,
\begin{equation}\label{eq3242.850}
\left ( \frac{2}{p}\right )=(-1)^{\frac{p^2-1}{8}}.
\end{equation}
\end{lem}
\begin{proof}[\textbf{Proof}] A detailed proof of the quadratic reciprocity laws appears in {\color{blue}\cite[Theorem 1.5.]{RH1994}}, and similar references.
\end{proof}
\begin{lem} \label{eq3242.860}  {\normalfont (Quadratic reciprocity law)} If $p$ and $q$ are odd primes, then
\begin{equation}\label{eq3242.870}
\left ( \frac{p}{q}\right )\left ( \frac{q}{p}\right )=(-1)^{\frac{p-1}{2}\frac{q-1}{2}}.
\end{equation}
\end{lem}
\begin{proof}[\textbf{Proof}]A detailed proof of the quadratic reciprocity laws appears in {\color{blue}\cite[Theorem 2.1.]{RH1994}}, and similar references.
\end{proof}

\section{Very Short Primitive Roots Tests}  \label{S3252.200-V}\hypertarget{S3252.200-V}
The set of Fermat primes 
\begin{equation}\label{lem3252.150}
\mathcal{F}=\left \{F_n=2^{2^n}+1:n\geq0\right \}= \left \{3,5,17,257,65537,\ldots\right \}
\end{equation}
has the simplest primitive root test: A quadratic nonresidue $q\geq3$ is a primitive root mod $F_n$. This follows from \hyperlink{S3242.500}{Lemma} \ref{lem3242.500}. The next set of primes with a short primitive root test seems to be the set of generalized Germain primes.
\begin{dfn} \label{dfn3252.500}\hypertarget{dfn3252.500} {\normalfont Let $s\geq 1$ be a parameter. The set of generalized Germain primes is defined by 
\begin{equation} \label{eq3252.340}
\mathcal{G}_s=\{p=2^s\cdot r+1: p \text{ and } r \text{ are primes}\}=\{3, 5, 7,13,17,23,29,37,43,41,73,\ldots\}\nonumber.
\end{equation}
}
\end{dfn}

\begin{lem} \label{lem3252.550}\hypertarget{lem3252.550} An integer $q\ne\pm1, n^2$ is a primitive root modulo a Germain prime $p=2^s\cdot r+1$ if and only if 
\begin{multicols}{2}
\begin{enumerate}[font=\normalfont, label=(\roman*)]
\item $\displaystyle  q^{2^{s-1}r}\not \equiv 1 \mod  p$,      
\item $\displaystyle q^{2^{s}}\not \equiv 1 \mod  p.$
\end{enumerate}
\end{multicols}
\end{lem}

\begin{proof}[\textbf{Proof}] Let $p=2^sr+1$ be a Germain prime, where $r\geq 2$ is prime, and $s\geq1$ is an integer. Since the totient $p-1=2^sr$ has two prime divisors, an integer $q\ne\pm1, n^2$ is a primitive root modulo $p$ if and only if 
\begin{enumerate}[font=\normalfont, label=(\roman*)]
\item $\displaystyle  q^{(p-1)/2}=q^{2^{s-1}r}\not \equiv 1 \mod  p$,
\item $\displaystyle q^{(p-1)/r} =q^{2^{s}}\not \equiv 1 \mod  p,$
\end{enumerate}
see \hyperlink{lem3242.500}{Lemma} \ref{lem3242.500}.
\end{proof}

As the Germain primes, the set of primes of the form
\begin{equation}
\mathcal{A}=\{p=k\cdot 2^n+1:n\geq0\},    
\end{equation}
with $k\geq3$ a fixed prime, have a very short primitive root test, similar to the algorithm in \hyperlink{lem3242.550}{Lemma} \ref{lem3252.550}. Some of these primes are factors of Fermat numbers. There are many interesting problems associated with these primes, a large literature, and numerical data, see \cite{KW1983}, et cetera.

\section{Density of Primitive Roots} \label{S3242.300-D}\hypertarget{S3242.300-D}
The formula for computing the density $\delta(u) \geq 0$ of primes with a fixed primitive root $u\ne \pm 1,v^2$ was conjectured by Artin and proved by Hooley in {\color{blue}\cite[p. 220]{HC1967}}, some historical perspective is given in \cite{SP2003}. The complete statement of the Artin-Hooley density formula for primitive roots mod $p$ is recorded here.  
\begin{thm} \label{thm3252.300D}\hypertarget{thm3252.300D} Let $u=(st^2)^k \ne \pm 1, v^2$ with $s$ squarefree, and $k \geq 1$ is odd, and let 
\begin{equation}\label{eq3252.300D00}
	a_k(u)=
	\prod_{p \,\mid \,k }\left( 1-\frac{1}{p-1}\right)  \prod_{p\, \nmid\, k} \left ( 1 -\frac{1}{p(p-1)} \right) .
\end{equation}

\begin{enumerate}[font=\normalfont, label=(\roman*)]
	\item If $s \not \equiv 1 \bmod 4$, then, 
	\begin{equation} \label{eq3252.300D02}
		\delta(u)=a_k(u). 
	\end{equation}
	\item If $s \equiv 1 \bmod 4$, then,  
	\begin{equation} \label{eq3252.300D04}
		\delta(u)=
		\left (1-\mu(s) \prod_{\substack{p\, \mid \, s \\ p\, \mid\, k }}\frac{1}{p-2}\prod_{\substack{p\,\mid \,s \\ p \,\nmid\, k }}\frac{1}{p^2-p-1} \right )
		a_k(u). 
	\end{equation}
\end{enumerate}
\end{thm}
\begin{proof}[\textbf{Proof}] Let $\xi_n=e^{i2\pi/n}$ be an $n$th root of unity and let $[\Q(\xi_n):\Q]$ be the degree of the finite extension $\Q(\xi_n)$ of $\Q$. The evaluation of the series
\begin{equation}
		\delta(u)=\sum_{n\geq1}\frac{\mu(n)}{[\Q(\xi_n):\Q]}=\sum_{n\geq1}\frac{\mu(n)\gcd(n,u)}{n\varphi(n)}
\end{equation}
is explicated in {\color{blue}\cite[Section 6]{HC1967}} and {\color{blue}\cite[Section 4]{MP2004}}. Generalizations to other Kummer extensions and related topics are given in \cite{LH1977}, \cite{JS2022}, et alii.
\end{proof}

The average density $\overline{\delta(u)}=a_1$ of primes with a fixed primitive root $u\ne\pm1, v^2$ has the value
\begin{equation} \label{eq3252.300D08}
	a_1= \prod_{p \geq 2} \left ( 1 -\frac{1}{p(p-1)} \right) =0.3739558136192022880547280\ldots.
\end{equation}
This numerical approximation was computed in \cite{WJ1961}. This value coincides with the simplest cases for squarefree $u=s\geq 2$ such that $s\not\equiv 1\bmod 4$. In general, the range of values of the density satisfy the inequalities
\begin{align}
	\frac{84}{85}a_1\leq \delta(u)\leq \frac{6}{5}a_1,
\end{align} this is easily derived from \eqref{eq3252.300D02} and \eqref{eq3252.300D04}, see the work in \cite{PS2024} for finer analysis.

\section{Heuristic} \label{S3292.333-H}\hypertarget{S3292.333-H}
The Artin-Hooley density formula for primitive roots mod $p$, given in \hyperlink{thm3252.300D}{Theorem} \ref{thm3252.300D}, suggests that the conjectural upper bound splits into two cases depending on the congruence class $q\not \equiv 1\bmod 4$ or $q \equiv 1\bmod 4$ of the integer $q>1$. These two cases exhibit distinct algebraic and analytic properties
\begin{enumerate}
	\item \textbf{Disjoint Case $q \not\equiv 1 \bmod 4$.} The quadratic field $\mathbb{Q}(\sqrt{q})$ requires an even modulus, a factor of 4, to embed into the cyclotomic extension $\mathbb{Q}(\zeta_{|q|})$. It remains linearly disjoint from odd-degree cyclotomic extensions.
	\item \textbf{Entangled Case $q \equiv 1 \bmod 4$.} The quadratic field embeds directly into the strictly odd conductor cyclotomic extension $\mathbb{Q}(\zeta_{|q|})$. This creates an unexpected field overlap, reducing the extension field degree $[\mathbb{Q}(\zeta_n, \sqrt[n]{q}) : \mathbb{Q}]$ by a factor of $2$. This forces a correction factor in the density formula.
\end{enumerate}
In both cases $q \not \equiv 1 \bmod 4$ and $q \equiv 1 \bmod 4$, the primorial-type integers 
\begin{equation}\label{eq3252.333D06B}
	q=2\cdot3\cdot5\cdots p_n, \qquad q=3\cdot5\cdot7\cdots p_n,
\end{equation}
et cetera, seems to be the worst cases  --- have the largest least primitive roots, see \hyperlink{S3292.444-N}{Section} \ref{S3292.444-N} for some numerical data.\\

The upper bound can be estimated by considering the case $s \equiv 1 \bmod 4$. Let $q=(st^2)^k \ne \pm 1, v^2$ with $s$ squarefree, and $k \geq 1$ is odd, and let $q_0=(s_0t^2)^k \ne \pm 1, v^2$ with $s_0=3\cdot5\cdots p_n\equiv s \bmod 4$ squarefree, and $q_0\leq q$. The integer $q_0\equiv 1\bmod 4$, which satisfies the congruence class $q\equiv 1\bmod 4$, seems to be the worst case possible --- has the largest least prime with a fixed primitive root $q_0$. The number of these primes $p\leq x $ with these fixed primitive roots has the lower bound 
\begin{align}\label{eq3252.333D08B}
	N_q(x)&\asymp\frac{x}{\log x}\cdot \left (1-\mu(s) \prod_{\substack{p\, \mid \, s \\ p\, \mid\, k }}\frac{1}{p-2}\prod_{\substack{p\,\mid \,s \\ p \,\nmid\, k }}\frac{1}{p^2-p-1} \right ) .
\end{align}
Dropping the inner finite product and further simplifications reduce it to
 
\begin{align}\label{eq3252.333D10B}
	N_q(x)
	&\gg\frac{x}{\log x}\cdot \left (1-\prod_{p\,\mid \,s}\frac{1}{p^2-p-1} \right ) \nonumber\\[.3cm]
	&\gg\frac{x}{\log x}\cdot \left (1-\prod_{p\leq s_0}\frac{1}{p^2-p-1} \right ) \nonumber\\[.3cm]
	&\gg \frac{x}{\log x}\cdot \prod_{2<p \leq s_0 }\left( 1-\frac{1}{p-1}\right)   \nonumber\\[.3cm]
	&\gg 1.
\end{align}
Rearranging and simplifying the finite product yield
\begin{align}\label{eq3252.333D12B}
	\frac{x}{\log x}&\gg \prod_{2<p \leq s_0 }\left( 1-\frac{1}{p-1}\right)^{-1}   \nonumber\\[.3cm]
	&\gg 	\log s_0.
\end{align}
This follows from Mertens theorem $\prod_{p \leq z}(1 - 1/p)^{-1}=(1+o(1))e^\gamma \log z$. Conclude that if there exist some primes $p\leq x$ such that 
\begin{equation}\label{eq3252.333D14B}
	x\gg (\log s_0)(\log\log s_0)\gg p,
\end{equation}
then $q_0$ is a primitive root modulo $p$ for some 
\begin{equation}
	p\ll  (\log s_0)(\log\log s_0)\leq (\log q_0)(\log\log q_0)
\end{equation} uniformly in $t\geq 1$ and odd $k\geq1$.\\

In light of these observations, \hyperlink{conj3232.000D}{Conjecture} \ref{conj3232.000D} is reformulated in the following way.

\begin{conj}   \label{conj3232.000R}\hypertarget{conj3232.000R} Let $q=(st^2)^k \ne \pm 1, v^2$ be a fixed nonnegative integer, with $s$ squarefree, and $k \geq 1$ is odd. Then, there exists an odd prime 
\begin{align}
p\ll (\log |q|)(\log \log |q|)
\end{align}	
such that $q$ is a primitive root modulo $p$.
\end{conj} 
A cursory inspection of formulas \eqref{eq3252.300D02} and \eqref{eq3252.300D04} reveals that there is a lot to learn about the sequences of $q$'s with interesting properties. Let $2,3,\ldots, p_n$ be the sequence of primes in increasing order and let the fixed primitive root be
\begin{equation}\label{eq3292.333H08}
	q_n=q^{3\cdot5\cdots p_n},
\end{equation}
where $q=2,3,5,6,\ldots $ is a fixed base. Many of the sequences defined in \eqref{eq3292.333H08} have bounded $p\ll1$. For the first few bases are listed in \autoref*{table-300-PE}.
\begin{table}[H]
	\setlength{\tabcolsep}{.750709cm}
	\renewcommand{\arraystretch}{1.250092}
	\setlength{\arrayrulewidth}{.975pt}
	\centering
	\begin{tabular}{c|c|c|c}
		 $q_n=q^{3\cdot5\cdots p_n}$  & Least Prime $p$ & $\log q$ & Ratio $p/\log q $ \\ \hline 
		$2^{3\cdot5\cdots p_n}$ &3 & 0.6931 & 4.3281\\ \hline 
	$3^{3\cdot5\cdots p_n}$ &5 & 1.0986 & 4.5512\\ \hline 
		$5^{3\cdot5\cdots p_n}$ &3 & 1.6094 & 1.8640\\ \hline
				$7^{3\cdot5\cdots p_n}$ &5 & 1.9459 & 2.5695\\ \hline
	\end{tabular}        
	\caption{Odd Primorial Exponent Primitive Root $q_n$}           
	\vskip .1 in		
	\label{table-300-PE}
\end{table}	
A result such as $p=3$ is the least prime for all primitive roots $q_n=2^{3\cdot5\cdots p_n}$ is closely analogous to or the dual to the classical result that states that if $g>1$ is a primitive root $\bmod\, p$ and $\bmod\, p^2$, then $g$ is a primitive root $\bmod\, p^n$ for all $n\geq1$, which is attributed to Lucas \cite{LE1878}.\\


\section{Representation of the Characteristic Function} \label{S3242.200-C}\hypertarget{S3242.200-C}
The standard representation of the characteristic function of primitive elements in finite fields stated below is divisor dependent.

\begin{lem} \label{lem3242.900}\hypertarget{lem3242.900}
Let \(G\) be a finite cyclic group of order \(p-1=\# G\), and let \(0\neq u\in G\) be an invertible element of the group. Then
\begin{equation}
\Psi (u)=\frac{\varphi (p-1)}{p-1}\sum _{d \mid p-1} \frac{\mu (d)}{\varphi (d)}\sum _{\ord(\chi ) = d} \chi (u)=
\left \{\begin{array}{ll}
1 & \text{ if } \ord_p (u)=p-1,  \\
0 & \text{ if } \ord_p (u)\neq p-1. \\
\end{array} \right .
\end{equation}
\end{lem}

\begin{proof}[\textbf{Proof}] This is widely available in the literature. A complete proof appears in Lemma 3.2 in \cite{CN2017}.
\end{proof}

The works in \cite{DH1937}, and \cite{WR2001} attribute the above formula to Vinogradov \cite{VI1930} and some authors attribute the above formula to Landau \cite{LE1927}. The proof and other details on the characteristic function are given in {\color{blue}\cite[p. 863]{ES1957}}, {\color{blue}\cite[p.\ 258]{LN1997}}, {\color{blue}\cite[p.\ 18]{MP2007}}. The characteristic function for multiple primitive roots is used in {\color{blue}\cite[p.\ 146]{CZ1998}} to study consecutive primitive roots. In \cite{DS2012} it is used to study the gap between primitive roots with respect to the Hamming metric. And in \cite{WR2001} it is used to prove the existence of primitive roots in certain small subsets \(A\subset \mathbb{F}_p\). In \cite{DH1937} it is used to prove that some finite fields do not have primitive roots of the form $a\tau+b$, with $\tau$ primitive and $a,b \in \mathbb{F}_p$ constants. In addition, the Artin primitive root conjecture for polynomials over finite fields was proved in \cite{PS1995} using this formula.\\

The result in the next result provides a \textit{divisor-free} of the representation of the characteristic function of primitive elements in finite fields $\F_p$. 
\begin{lem} \label{lem3242.300A}\hypertarget{lem3242.300A}
Let \(p\geq 2\) be a prime, and let \(\tau\) be a primitive root mod \(p\). If \(u\in\mathbb{F}_p\) is a nonzero element, then
\begin{equation}\label{eq3242.300A00}
\Psi (u)=\sum _{\substack{1\leq s\leq p-1\\\gcd (s,p-1)=1}} \frac{1}{p}\sum _{0\leq m\leq p-1} e^{\frac{i2 \pi \left (\tau ^s-u\right)m}{p}}
=\left \{
\begin{array}{ll}
1 & \text{ if } \ord_p (u)=p-1,  \\
0 & \text{ if } \ord_p (u)\neq p-1. \\
\end{array} \right .\nonumber
\end{equation}
\end{lem}
\begin{proof}[\textbf{Proof}] Set the additive character $\psi(s) =e^{i 2\pi  as/p}\in \C$. Now, as the index $s\in \mathscr{R}$ ranges over the integers relatively prime to $\varphi(p)=p-1$, the element $\tau ^s\in \F_p ^{\times}$ ranges over the primitive roots modulo $p$. Accordingly, the equation 
	\begin{eqnarray}\label{eq9977FF.200DF2}
		a=\tau ^s- u=0
	\end{eqnarray}has a unique solution $s\in \mathscr{R}$ if and only if the fixed element $u\in \F_p$ is a primitive root. This implies that the inner sum in 	
	\begin{equation}\label{eq9977FF.200DF}
		\sum _{\substack{1\leq s\leq p-1\\\gcd (s,p-1)=1}} \frac{1}{p}\sum _{0\leq t< p} e^{i 2\pi \frac{(\tau ^s-u)t}{p}}=
		\left \{\begin{array}{ll}
			1 & \text{   \normalfont if } \ord_{p} (u)=p-1,  \\
			0 & \text{   \normalfont if } \ord_{p} (u)\ne p-1. \\
		\end{array} \right.
	\end{equation} 
	collapses to $\sum _{0\leq s< p} e^{i 2\pi as/p}=\sum _{0\leq s< p} 1=p $. Otherwise, if the element $u\in \F_p$ is not a primitive root, then the equation \eqref{eq9977FF.200DF2} has no solution $n\in \mathscr{R}$, and the inner sum in \eqref{eq9977FF.200DF} collapses to $\sum _{0\leq s< p} e^{i 2\pi as/p}=0$,
	this follows from the geometric series formula $\sum_{0\leq n\leq  N-1} w^n =(w^N-1)/(w-1)$, where $w=e^{i 2\pi a/(p-1)}\ne1$ and $N=p$. 
	This completes the verification.	 
\end{proof}
A sum $\sum_{p\leq x}\Psi (u)$ of the plain characteristic function in \hyperlink{lem3242.300A}{Lemma} \ref{lem3242.300A} allocates the correct density of primitive root whenever $u\ne\pm, v^2$ satisfies the conditions $\mu(u)\ne0$ and $u\not\equiv 1\bmod 4$. Otherwise the density is not correct --- it requires a correction factor. A different version that allocate the correct density for all cases $u\ne\pm, v^2$ is stated below. 

\begin{lem} \label{lem3242.300W}\hypertarget{lem3242.300W} Let \(p\geq 2\) be a prime, and let $u=(st^2)^k \ne \pm 1, v^2$ with $s$ squarefree, and $k \geq 1$, and let \(\tau\) be a primitive root mod \(p\). Define the weight function	
\begin{equation}
w(u)=\delta(u)\prod_{p\geq2} \left ( 1 -\frac{1}{p(p-1)} \right)^{-1}.
\end{equation} 
Then
\begin{equation}\label{eq3242.300W00}
		\Psi_w (u)=\sum _{\substack{1\leq s\leq p-1\\\gcd (s,p-1)=1}} \frac{w(u)}{p}\sum _{0\leq m\leq p-1} e^{\frac{i2 \pi \left (\tau ^s-u\right)m}{p}}
		=\left \{
		\begin{array}{ll}
			w(u) & \text{ if } \ord_p (u)=p-1,  \\
			0 & \text{ if } \ord_p (u)\neq p-1. \\
		\end{array} \right .\nonumber
	\end{equation}
\end{lem}
\begin{proof}[\textbf{Proof}] The weight function is derived from the density result in \hyperlink{thm3252.300D}{Theorem} \ref{thm3252.300D} and the structure of the double sum is the same as that in \hyperlink{lem3242.300A}{Lemma} \ref{lem3242.300A}.
\end{proof}


\section{Estimates of Exponential Sums} \label{S9933Q-A}\hypertarget{S9933Q-A}
Some basic estimates of exponential sums are recorded and proved here. 

\subsection{Finite Summation Kernel}
\begin{lemma}   \label{lem5555.400B}\hypertarget{lem5555.400B}  Let \(p\geq 2\) be large prime, and let $\omega=e^{i2 \pi/p} $ be a $p$th root of unity. Then,
	\begin{equation}
		\sum_{1 \leq t\leq p-1}\Bigg | \sum_{\substack{1\leq n\leq p-1\\\gcd(n,p-1)=1}} \omega^{tn}  \Bigg  |\ll  p^{1+\delta} \log p ,
	\end{equation}		where $\delta>0$ is a small number. 
\end{lemma} 

\begin{proof}[\textbf{Proof}] Use the inclusion exclusion principle to rewrite the exponential sum as
	\begin{eqnarray} \label{eq5555.400i}
		\sum_{\substack{1\leq n\leq p-1\\\gcd(n,p-1)=1}}\omega^{tn}&=& \sum_{n \leq p-1} \omega^{tn}  \sum_{\substack{d \mid p-1 \\ d \mid n}}\mu(d)  \\
		&=& \sum_{d \mid p-1} \mu(d) \sum_{\substack{n \leq p-1 \\ d \mid n}} \omega^{tn}\nonumber \\
				&=& \sum_{d \mid p-1} \mu(d) \frac{\omega^{dt}-\omega^{dt((p-1)/d+1)}}{1-\omega^{dt}} \nonumber.
	\end{eqnarray} 
	Now, the parameters are $p$ prime, $\omega=e^{i2 \pi/p}$, the integers $t \in [1, p-1]$ and $d \leq p-1<p$. This data implies that $\pi dt/p\ne k \pi $ with $k \in \mathbb{Z}$, so the sine function $\sin(\pi dt/p)\ne 0$ is well defined. Consequently, the absolute value satisfies
	\begin{equation}
		\left |\frac{\omega^{dt}-\omega^{dt((p-1)/d+1)}}{1-\omega^{dt}} \right |\leq 	\left | \frac{2}{\sin( \pi dt/ p)} \right |.
	\end{equation}
	For each $d\mid p-1$, the map $t\longrightarrow z\equiv dt \bmod p$ is a permutation in the finite field $\F_p$. Thus, using standard manipulations, and $z/2 \leq \sin(z) <z$ for $0<|z|<\pi/2$, the last expression becomes
	\begin{eqnarray}
		\sum_{1 \leq t\leq p-1}\Bigg | 	\sum_{\substack{1\leq n\leq p-1\\\gcd(n,p-1)=1}}\omega^{tn}\Bigg  |&\leq&\sum_{d \mid p-1,} \sum_{1 \leq t\leq p-1}	\left | \frac{2}{\sin( \pi dt/ p)} \right |\\
		&\leq&\sum_{d \mid p-1,} \sum_{1 \leq z\leq p-1}	 \frac{2p}{ \pi z} \nonumber\\
		&\ll&  p^{1+\delta} \log p \nonumber,
	\end{eqnarray}
	where $\sum_{d \mid p-1}1=d(p-1)\ll p^{\delta}$ is the number of divisor in $p-1$ and $\delta>0$ is a small number. 
\end{proof}
\subsection{Gauss-type Sum}
An elementary exponential sum estimate of a Gauss-type sum is provided in this section. 
\begin{lemma}   \label{lem1234A.150A}\hypertarget{lem1234A.150A}  
	Let \(p\geq 2\) be a prime, let $\chi(t)=e^{i2 \pi t/p} $ and  $\psi(t)=e^{i2\pi  \tau^t/p}$ be a pair of characters. Then, the Gaussian sum has the upper bound
	\begin{equation} \label{eq3-355}
		\left |\sum_{1 \leq t \leq p-1}    \chi(t) \psi(t) \right | \leq 4 p^{1/2} \log p.\nonumber
	\end{equation}
\end{lemma} 

\begin{proof}[\textbf{Proof}] This Gaussian type sum can be rewritten as
	\begin{equation} \label{eq3-355b}
		\sum_{1 \leq t \leq p-1}    \chi(t) \psi(t)=\sum_{1 \leq t \leq p-1}   e^{\frac{i\pi (at+b\tau^t)}{p}},
	\end{equation}	
	where $a,b\ne0$ are constants, has an upper bound completely independent of $a$ and $b$. 	
	The original proof using Legendre resolvent seems to be that in \cite{ML1972}, a second proof was recent given in \cite{CC2009}.
\end{proof}


\subsection{Estimates of the Power Exponential Sum} \label{S9933Q-AA}\hypertarget{S9933Q-AA}
The estimate for the power sum with relatively prime index is based on the identity
\begin{equation}\label{eq9933Q.210c}
	\frac{1}{p} \sum_{0 \leq t\leq p-1,}  \sum_{0 \leq s\leq p-1} \omega^{t(n-s)}f(s)=f(n),
	\end{equation}
where $\omega=e^{i2\pi/p}$ and $x  \leq p -1$. The case of interest here is the exponential sum
\begin{equation} \label{eq9933Q.346b}
	\sum_{\substack{1\leq n\leq p-1\\\gcd(n,p-1)=1}} e^{\frac{i2\pi b \tau^n}{p}}= \sum_{\substack{1\leq n\leq p-1\\\gcd(n,p-1)=1}}\frac{1}{p} \sum_{0 \leq t\leq p-1,}  \sum_{1 \leq s\leq p-1} \omega^{t(n-s)}e^{\frac{i2\pi b \tau^s}{p}}.
\end{equation}
Use the inclusion exclusion principle to rewrite the exponential sum as
\begin{equation}\label{eq9933Q.346d}
	\sum_{\substack{1\leq n\leq p-1\\\gcd(n,p-1)=1}} e^{ \frac{i2\pi b \tau^n}{p}} 
	= \sum_{1\leq  n \leq p-1}\frac{1}{p} \sum_{0 \leq t\leq p-1,}  \sum_{1 \leq s\leq p-1} \omega^{t(n-s)}e^{\frac{i2\pi b \tau^s}{p}} \sum_{\substack{d \mid p-1 \\ d \mid n}}\mu(d)   .
\end{equation} 
Substituting $t=0$ in \eqref{eq9933Q.210c} contributes $-\varphi(p-1)/p$, and rearranging it yield
\begin{eqnarray}\label{eq9933Q.346f}
	&& \sum_{\substack{1\leq n\leq p-1\\\gcd(n,p-1)=1}} e^{ \frac{i2\pi b \tau^n}{p}} \\
	&=& \sum_{1\leq  n \leq p-1}\frac{1}{p} \sum_{1 \leq t\leq p-1,}  \sum_{1 \leq s\leq p-1} \omega^{t(n-s)}e^{\frac{i2\pi b \tau^s}{p}} \sum_{\substack{d \mid p-1 \\ d \mid n}}\mu(d) -\frac{\varphi(p-1)}{p} \nonumber \\
	&=&\frac{1}{p} \sum_{1 \leq t\leq p-1} \left ( \sum_{1 \leq s\leq p-1} \omega^{-ts}e^{\frac{i2\pi b \tau^s}{p}}\right )\left (\sum_{d \mid p-1} \mu(d) \sum_{\substack{1\leq n \leq p-1, \\ d \mid n}}   \omega^{tn} \right ) -\frac{\varphi(p-1)}{p} \nonumber.
\end{eqnarray} 
\begin{thm}  \label{thm9933Q.346}\hypertarget{thm9933Q.346}  Let \(p\geq 2\) be a large prime, and let $\tau $ be a primitive root modulo $p$. Then,
	\begin{equation}
		\sum_{\substack{1\leq n\leq p-1\\\gcd(n,p-1)=1}} e^{i2\pi b \tau^n/p} \ll  p^{1/2+\varepsilon} \nonumber,
	\end{equation} 
	where $\varepsilon>0$ is a small real number and the implied constant is independent of $b\ne0$. 	
\end{thm}
\begin{proof}[\textbf{Proof}]Taking absolute value in \eqref{eq9933Q.346f} and applying the estimates provided in \hyperlink{lem5555.400B}{Lemma} \ref{lem5555.400B} and \hyperlink{lem1234A.150A}{Lemma} \ref{lem1234A.150A} complete the verification.
\end{proof}

There are many proofs of this result, consult \cite{ML1972}, {\color{blue}\cite[Theorem 6]{FS2000}}, and related results are given in \cite{FS2001}, \cite{GM2005}, \cite{CC2009}, and {\color{blue}\cite[Theorem 1]{GK2005}}. The upper bound given in \hyperlink{thm9933Q.346}{Theorem} \ref{thm9933Q.346} seems to be optimum. A different proof, which has a weaker upper bound, appears in {\color{blue}\cite[Theorem 6]{FS2000}}, and related results are given in \cite{CC2009}, \cite{FS2002}, \cite{GK2005}, and {\color{blue}\cite[Theorem 1]{GK2005}}.
\subsection{Equivalent Power Exponential Sums} 
For any fixed $ 0 \ne b \in \mathbb{F}_p$, the map $ \tau^n \longrightarrow b \tau^n$ is one-to-one (permutation) in $\mathbb{F}_p$. Consequently, the subsets 
\begin{equation} \label{eq9933EES.500E00}
	\{ \tau^n: \gcd(n,p-1)=1 \}\quad \text { and } \quad  \{ b\tau^n: \gcd(n,p-1)=1 \} \subset \mathbb{F}_p
\end{equation} have the same cardinalities. As a direct consequence the exponential sums 
\begin{equation} \label{eq9933EES.500E02}
	V(b)=\sum_{\substack{1\leq n\leq p-1\\\gcd(n,p-1)=1}}e^{i2\pi b \tau^n/p} \quad \text{ and } \quad V(1)=\sum_{\substack{1\leq n\leq p-1\\\gcd(n,p-1)=1}} e^{i2\pi \tau^n/p},
\end{equation}
have the same upper bound up to an error term. An asymptotic relation for the exponential sums \eqref{eq9933EES.500E02} is provided here. 

\begin{thm}   \label{thm9933EES.500E}\hypertarget{thm9933EES.500E}  Let \(p\geq 2\) be a large prime. If $\tau $ be a primitive root modulo $p$, then
\begin{equation} \label{eq9933RPI.500d}
\sum_{\substack{1\leq n\leq p-1\\\gcd(n,p-1)=1}} e^{\frac{i2\pi b \tau^n}{p}} =  \sum_{\substack{1\leq n\leq p-1\\\gcd(n,p-1)=1}} e^{\frac{i2\pi  \tau^n}{p}} + O(p^{1/2+\varepsilon} )\nonumber,
	\end{equation} 
	where $\varepsilon>0$ is a small number and the implied constant is independent of $b \in [1, p-1]$. 	
\end{thm} 
\begin{proof}[\textbf{Proof}] For $b\in[1,p-1]$, the exponential sum has the representation 
	\begin{eqnarray} \label{eq9933RPI.500f}
		V(b)&=& \sum_{\substack{1\leq n\leq p-1\\\gcd(n,p-1)=1}} e^{\frac{i2\pi b \tau^n}{p}} \\
		&=&\frac{1}{p} \sum_{1 \leq t\leq p-1} \left ( \sum_{1 \leq s\leq p-1} \omega^{-ts}e^{\frac{i2\pi b \tau^s}{p}}\right )\left (\sum_{d \mid p-1} \mu(d) \sum_{\substack{n \leq p-1, \\ d \mid n}}   \omega^{tn} \right ) -\frac{\varphi(p-1)}{p}\nonumber,
	\end{eqnarray} 
	confer equations \eqref{eq9933Q.346b}, \eqref{eq9933Q.346d} and \eqref{eq9933Q.346f} for more details. In particular, for $b=1$, 
	\begin{eqnarray} \label{eq9933RPI.500h}
		V(1)&=& 	\sum_{\substack{1\leq n\leq p-1\\\gcd(n,p-1)=1}}e^{\frac{i2\pi  \tau^n}{p}} \\
		&=& \frac{1}{p} \sum_{1 \leq t\leq p-1} \left ( \sum_{1 \leq s\leq p-1} \omega^{-ts}e^{\frac{i2\pi a \tau^s}{p}}\right )\left (\sum_{d \mid p-1} \mu(d) \sum_{\substack{n \leq p-1, \\ d \mid n}}   \omega^{tn} \right ) -\frac{\varphi(p-1)}{p}\nonumber,
	\end{eqnarray}
	respectively. Differencing \eqref{eq9933RPI.500f} and \eqref{eq9933RPI.500h} produces 
	\begin{eqnarray} \label{eq9933RPI.500i}
		V(b)-V(1)&= &	\sum_{\substack{1\leq n\leq p-1\\\gcd(n,p-1)=1}} e^{\frac{i2\pi b \tau^n}{p}} -\sum_{\substack{1\leq n\leq p-1\\\gcd(n,p-1)=1}} e^{\frac{i2\pi  \tau^n}{p}} \\
		&=&     \frac{1}{p} \sum_{1 \leq t\leq p-1} \left ( \sum_{1 \leq s\leq p-1} \omega^{-ts}e^{\frac{i2\pi  b \tau^s}{p}}-\sum_{1 \leq s\leq p-1} \omega^{-ts}e^{\frac{i2\pi  \tau^s}{p}}\right ) \nonumber \\
		&& \hskip 2.5 in \times \left (\sum_{d \mid p-1} \mu(d) \sum_{\substack{n \leq p-1, \\ d \mid n}}   \omega^{tn} \right ) \nonumber.
	\end{eqnarray}
Taking absolute value of the difference $D(b)=V(b)-V(1)$ and applying the estimates provided in \hyperlink{lem5555.400B}{Lemma} \ref{lem5555.400B} and \hyperlink{lem1234A.150A}{Lemma} \ref{lem1234A.150A} return 
	
	\begin{eqnarray} \label{eq9933RPI.500j}
		|D(b)|&=&\left| \sum_{\substack{1\leq n\leq  p-1\\\gcd(n,p-1)=1}} e^{\frac{i2\pi b \tau^n}{p}} -\sum_{\substack{1\leq n\leq  p-1\\\gcd(n,p-1)=1}} e^{\frac{i2\pi  \tau^n}{p}}\right|  \\[.3cm]
		&\leq&  \frac{1}{p}   \sum_{1 \leq t\leq p-1} \left | \sum_{1 \leq s\leq p-1} \omega^{-ts}e^{\frac{i2\pi b \tau^s}{p}}-\sum_{1 \leq s\leq p-1} \omega^{-ts}e^{\frac{i2\pi  \tau^s}{p}}\right | \nonumber \\[.3cm]
		&&\hskip 2.75in \times \left |\sum_{d \mid p-1} \mu(d) \sum_{\substack{n \leq p-1, \\ d \mid n}}   \omega^{tn} \right | \nonumber.
	\end{eqnarray}
By \hyperlink{lem5555.400B}{Lemma} \ref{lem5555.400B}, the relatively prime summation kernel is bounded by
	\begin{eqnarray} \label{eq9933RPI.500m}
		\sum_{1 \leq t\leq p-1}	\Bigg |\sum_{d \mid p-1} \mu(d) \sum_{\substack{n \leq p-1, \\ d \mid n}}   \omega^{tn} \Bigg | 
		&=& \sum_{1 \leq t\leq p-1}\Bigg | \sum_{\gcd(n, p-1)=1}\omega^{tn} \Bigg |  \\ 
		&\ll &  p^{1+\delta}\log p\nonumber, 
	\end{eqnarray}
	where $\delta>0$ is a small number and by \hyperlink{lem1234A.150A}{Lemma} \ref{lem1234A.150A}, the difference including Gauss-type sums is bounded by
	\begin{eqnarray} \label{eq9933RPI.500o}
		\Bigg |  \sum_{1 \leq s\leq p-1} \omega^{-ts}e^{\frac{i2\pi b  \tau^s}{p}}-  \sum_{1 \leq s\leq p-1} \omega^{-ts}e^{\frac{i2\pi  \tau^s}{p}}\Bigg |&\leq& 2	\Bigg | \sum_{1 \leq s\leq p-1} \chi(s) \psi(s) \Bigg| \nonumber\\[.3cm]
		&\leq &8 p^{1/2} \log p, 
	\end{eqnarray}
Replacing \eqref{eq9933RPI.500m} and \eqref{eq9933RPI.500o} return
\begin{eqnarray} \label{eq9933RPI.500p}
|D(b)|	&=&\frac{1}{p}\sum_{1 \leq t\leq p-1}\Bigg | \sum_{1 \leq s\leq p-1} \omega^{-ts}e^{\frac{i2\pi b \tau^s}{p}}-  \sum_{1 \leq s\leq p-1} \omega^{-ts}e^{\frac{i2\pi  \tau^s}{p}} \Bigg | \cdot \Bigg | \sum_{d \mid p-1} \mu(d) \sum_{\substack{n \leq p-1, \\ d \mid n}}   \omega^{tn} \Bigg | \nonumber\\[.3cm]
		&\ll &\frac{1}{p}(8 p^{1/2} \log p)\cdot \sum_{1 \leq t\leq p-1}\Bigg | \sum_{d \mid p-1} \mu(d) \sum_{\substack{n \leq p-1, \\ d \mid n}}   \omega^{tn} \Bigg | \nonumber\\[.3cm]
		&\ll &\frac{1}{p}(8 p^{1/2} \log p)\cdot (p^{1+\delta} \log p) \nonumber\\[.3cm]
		&\ll & p^{1/2+\varepsilon} ,
	\end{eqnarray}
	where $\varepsilon>\delta>0$, the implied constant depends on the number of divisors of $p-1$ and it is independent of $b\in [1,p-1]$.
\end{proof}

\section{Evaluation of the Main Term} \label{S3262.300-M}\hypertarget{S3262.300-M}
The precise evaluation of the main term $M(x)$ occurring in the proof of Theorem \ref{thm3232.100} is recorded here. The symbol $\li(x)$ denotes the logarithm integral.
\begin{lem} \label{lem3262.200}\hypertarget{lem3262.200} {\normalfont {(\color{blue}\cite[Lemma 1]{SP1969}}) } Let \(x\geq 1\) be a large number, and let \(\varphi (n)\) be the Euler totient function. Then
	\begin{equation}\label{eq3262.200A00}
		\sum_{1\leq p\leq x }\frac{\varphi(p-1)}{p-1} =a_1\li(x) +
		O\left(\frac{x}{(\log x)^c}\right) ,
	\end{equation}
	where \(\li(x)\) is the logarithm integral, $a_1=0.373955\ldots $, and \(c> 1\) is an arbitrary constant, as \(x \rightarrow \infty\). 
\end{lem}
\begin{lem}  \label{lem3262.300}\hypertarget{lem3262.300}  Let \(x\geq 1\) be a large number, and let $p$ be a prime. Then,
	\begin{equation}\label{eq3262.300A00}
	\sum_{1\leq p\leq x} \frac{w(u)}{p}\sum_{\substack{1\leq n\leq x\\\gcd(n,p-1)=1}} 1=\delta(u)\li(x)+O\left(\frac{x}{\log
		^bx}\right) ,
	\end{equation} 
	where $c>1$ is a constant.
\end{lem}
\begin{proof}[\textbf{Proof}] Simplifying the inner sum and applying  \hyperlink{lem3262.200}{Lemma} \ref{lem3262.200} yield
	\begin{align}\label{eq3262.300A02}
	\sum_{1\leq p\leq x} \frac{w(u)}{p}\sum_{\substack{1\leq n\leq x\\\gcd(n,p-1)=1}} 1&=w(u)\sum_{1\leq p\leq x} \frac{\varphi(p-1)}{p-1}+O\left( \sum_{p\leq x} \frac{1}{p}\right) \\[.3cm]
	&=w(u)\cdot a_1\li(x)+O\left(\frac{x}{\log
		^cx}\right)\nonumber\\[.3cm]
		&=\delta(u)\cdot \li(x)+O\left(\frac{x}{\log
		^cx}\right) \nonumber.
\end{align}	
The constant term is precisely	
\begin{equation} \label{eq3262.020}
	w(u)\cdot a_1=\delta(u)\prod_{p \geq 2 } \left(1-\frac{1}{p(p-1)}\right)^{-1}\cdot a_1=\delta(u). 
\end{equation} 	
Thus, the claim follows. 
\end{proof}

\section{Estimate for the Error Term} \label{S3282.500-E}\hypertarget{S3282.500-E}
The upper bounds for the error term $E(x)$ in the proof of Theorem \ref{thm3232.100} is recorded here. 

\begin{lem} \label{lem3282.500}\hypertarget{S3282.500}  Let \(p\geq 2\) be a large prime, and let \(\tau\) be a primitive root mod \(p\). If the element \(u\ne 0\) is not a primitive root, then, 
	\begin{equation} \label{eq3282.020}
	\sum_{1 \leq p\leq x}
	\frac{1}{p}\sum_{\substack{1\leq n\leq p-1\\\gcd(n,p-1)=1}} \sum_{ 0<k\leq p-1} e^{\frac{i2 \pi (\tau ^n-u)s}{p}}\ll  x^{1/2+\varepsilon}\nonumber
	\end{equation} 
where $\varepsilon>0$ is an arbitrarily small number, for all sufficiently large numbers $x\geq 1$.
\end{lem}

\begin{proof}[\textbf{Proof}] Rearranging the order of summation and applying \hyperlink{thm9933EES.500E}{Theorem} \ref{thm9933EES.500E} yield
\begin{align} \label{eq3282.022}
E(x)&=	\sum_{1 \leq p\leq x} \frac{1}{p}\sum_{\substack{1\leq n\leq p-1\\\gcd(n,p-1)=1}} \sum_{ 0<s\leq p-1} e^{\frac{i2 \pi (\tau ^n-u)s}{p}}\nonumber\\[.3cm]
&=\sum_{1 \leq p\leq x}\frac{1}{p} \sum_{ 0<s\leq p-1}e^{\frac{-i2 \pi us}{p}}\sum_{\substack{1\leq n\leq p-1\\\gcd(n,p-1)=1}} e^{\frac{i2 \pi \tau ^n s}{p}}\nonumber\\[.3cm]
&=\sum_{1 \leq p\leq x}\frac{1}{p} \sum_{ 0<s\leq p-1}e^{\frac{-i2 \pi us}{p}}\left( \sum_{\substack{1\leq n\leq p-1\\\gcd(n,p-1)=1}} e^{\frac{i2 \pi \tau ^n }{p}}+O(p^{1/2+\delta})\right) ,
	\end{align}
where $\delta>0$. By the triangle inequality and applying \hyperlink{thm9933Q.346}{Theorem} \ref{thm9933Q.346} to the inner sum yield the upper bound
\begin{align} \label{eq3282.024}
 \Bigg |\sum_{\substack{1\leq n\leq p-1\\\gcd(n,p-1)=1}} e^{\frac{i2 \pi \tau ^n }{p}}+O(p^{1/2+\delta})\Bigg |
	&\leq  \Bigg |\sum_{\substack{1\leq n\leq p-1\\\gcd(n,p-1)=1}} e^{\frac{i2 \pi \tau ^n }{p}}\Bigg |+\Bigg | O(p^{1/2+\delta})\Bigg |\nonumber\\[.3cm]
	&\leq c_0\,p^{1/2+\delta})+ c_1\,p^{1/2+\delta})\nonumber\\[.3cm]
	&\leq c_2 p^{1/2+\varepsilon}\nonumber\\[.3cm]
	&\ll x^{1/2+\varepsilon},	
\end{align}
where the constants $c_0,c_1,c_2>0$ are independent of $s\in [1,p-1]$ and $\varepsilon>\delta>0$.   
Taking absolute value  
\begin{align} \label{eq3282.026}
E(x)&\leq\sum_{1 \leq p\leq x}\frac{1}{p}\cdot \Bigg | \sum_{ 0<s\leq p-1}e^{\frac{-i2 \pi us}{p}}\Bigg |\cdot \Bigg |\sum_{\substack{1\leq n\leq p-1\\\gcd(n,p-1)=1}} e^{\frac{i2 \pi \tau ^n }{p}}+O(p^{1/2+\delta})\Bigg |\nonumber\\[.3cm]
&\ll\sum_{1 \leq p\leq x}\frac{1}{p} \cdot 1\cdot p^{1/2+\varepsilon}\nonumber\\[.3cm]
&\ll x^{1/2+\varepsilon}.	
\end{align} 
This completes the proof.
\end{proof}

\section{Main Result} \label{S3292.900-T}\hypertarget{S3292.900-T}
Let $\mathbb{P}=\{2,3,5,\ldots, \}$ denotes the set of primes and let $\mathcal{N}_q=\{ p \leq x: \ord_p(q)=p-1 \}\subset \mathbb{P}$ denotes the subset of primes with a fixed primitive root $q\ne \pm 1,n^2$. Given a large number $x\geq1$, the precise primes counting functions are defined by 
\begin{equation} \label{eq3292.900}
\pi(x)=\# \{ p \leq x: p \text{ is prime} \},
\end{equation}
and
\begin{equation} \label{eq3292.910}
N_{q}(x)=\# \{ p \leq x: \ord_p(q)=p-1 \}.
\end{equation}
The density of the subset of primes $\mathcal{N}_q$ is defined by the limit
\begin{equation} \label{eq3292.920}
\delta(q)=\lim_{x \to \infty} \frac{N_{q}(x)}{\pi(x)}=c(q)a_1.
\end{equation}
The constant $a_1>0$ is defined in \eqref{eq3262.020}, and the rational number $c(q)\geq 0$ is the correction factor, see \hyperlink{thm3252.300D}{Theorem} \ref{thm3252.300D}.

\begin{proof} [{\color{blue}\textbf{Proof }\bf  Theorem \ref{thm3232.100}}] Let $q \ne \pm 1,n^2$ be a fixed integer and let $(\log |q|)(\log \log |q|)\ll x$ be a real number. Suppose that the fixed integer $q$ is not a primitive root for all primes short interval $p \in [1,x]$. Summing of the characteristic function of primitive roots in the prime finite field $\mathbb{F}_p$ over the interval $[1,x]$ returns the nonexistence equation
\begin{equation} \label{eq3292.940}
	0=\sum_{1 \leq p\leq x} \Psi (q).
\end{equation}
Replacing the characteristic function, \hyperlink{S3242.300W}{Lemma} \ref{lem3242.300W}, and expanding the nonexistence equation \eqref{eq3292.940} yield
\begin{eqnarray} \label{eq3292.950}
N_{q}(x)&=&\sum_{1 \leq p\leq x} \Psi_w (q) \\[.3cm] 
	&=&\sum_{1 \leq p\leq x} \left (\frac{w(u)}{p}\sum_{\substack{1\leq n\leq p-1\\\gcd(n,p-1)=1}} \sum_{ 0\leq s\leq p-1} e^{\frac{i2 \pi  (\tau ^n-q)s}{p}} \right ) \nonumber\\[.3cm]
	&=&\sum_{1 \leq p\leq x}\frac{w(u)}{p}\sum_{\substack{1\leq n\leq p-1\\\gcd(n,p-1)=1}} 1  +\sum_{1 \leq p\leq x}
	\frac{w(u)}{p}\sum_{\substack{1\leq n\leq p-1\\\gcd(n,p-1)=1}} \sum_{ 0<s\leq p-1}e^{\frac{i2 \pi  (\tau ^n-q)s}{p}}\nonumber\\[.3cm]
	&=&M(x) + E(x)\nonumber,
	\end{eqnarray} 
	where $q\geq 2$ is a fixed integer. \\
	
	The main term $M(x)$ is determined by a finite sum over the trivial additive character \(\psi(s) =e^{i 2\pi  st/p}=1 \) for $s\ne0$, and the error term $E(x)$ is determined by a finite sum over the nontrivial additive characters \(\psi(s) =e^{i 2\pi s t/p}\ne1 \) for $t\ne0$.\\
	
Applying \hyperlink{S3262.300}{Lemma} \ref{lem3262.300} to the main term, and \hyperlink{S3282.500}{Lemma} \ref{lem3282.500} to the error term yield
	\begin{eqnarray} \label{eq3292.960}
	N_{q}(x)&=&\sum_{1 \leq p\leq x} \Psi (q) \nonumber \\[.3cm]
	&=&M(x) + E(x) \nonumber\\[.3cm]
	&=&\delta(u)\cdot \li(x) +O\left(\frac{x}{(\log x)^c}\right)+O\left( x^{1/2+\varepsilon} \right) \nonumber\\[.3cm]
	&=&\delta(u)\cdot \li(x)+O\left( \frac{x }{(\log x)^c} \right)  ,
	\end{eqnarray} 
	where $\delta(q)>0$ is the density of the primitive root $q$ defined in \hyperlink{thm3252.300D}{Theorem} \ref{thm3252.300D}, $c>1$ is a constant and $\varepsilon>0$ is a small number. Since logarithm integral
\begin{equation}
 \li(x)\gg \frac{x}{\log x}>0   
\end{equation}
for all real numbers $x\gg(\log |q|)(\log \log |q|)>2$, the expression
\begin{eqnarray} \label{eq3292.970}
N_{q}(x)&=&\sum_{1\leq p\leq x} \Psi_w (q)  \\[.3cm]
	&=&\delta(q)\cdot\li(x)+O\left( \frac{x }{(\log x)^c} \right)  \nonumber\\[.3cm]
	&\gg& \frac{x}{\log x}\nonumber\\[.3cm]
    &>&0,\nonumber
	\end{eqnarray} 
contradicts the hypothesis \eqref{eq3292.940} for all sufficiently large numbers \begin{equation}
	2< (\log |q|)(\log \log |q|)\ll x.
\end{equation} 
Ergo, this proves the existence of a small prime $p\in [1,x]$ with a fixed primitive root $q$.  
\end{proof}

\section{Numerical Data} \label{S3292.444-N}\hypertarget{S3292.444-N}  
Some numerical data for some of the expected extreme cases of the integer pairs $q$ and $p$are compiled here. Three important collections of integers are tested here to verify the predicted data.\\

The small sample of colossally abundant integers in \autoref{table-300-C} has a slightly more complicated pattern. Nevertheless, it shows that the least primes $p$ with a fixed colossally abundant primitive root $q=2^{v_2}\cdot3^{v_3}\cdots p_k$ is about $p\approx p_k\leq c(p,q)(\log q)\ll (\log q)(\log\log q) $. Assuming this approximation is correct, the determination of the constant $c(p,q)$ is an interesting  problem.
\begin{table}[H]
	\setlength{\tabcolsep}{.49cm}
	\renewcommand{\arraystretch}{1.250092}
	\setlength{\arrayrulewidth}{.975pt}
	\centering
\begin{tabular}{c|c|c|c} 
Primitive Root $q$ & Least Prime $p$ & $(\log q)(\log\log q)$ &  $\frac{p}{(\log q)(\log\log q)}$ \\ \hline 
	2 & 3 & -0.2541 & -11.8086 \\ \hline 
	6 & 11 & 1.0450 & 10.5268 \\ \hline 
	12 & 5 & 2.2618 & 2.2106 \\ \hline 
	60 & 19 & 5.7714 & 3.2921 \\ \hline 
	120 & 23 & 7.4972 & 3.0678 \\ \hline 
	360 & 7 & 10.4337 & 0.6709 \\ \hline 
	2520 & 13 & 16.1200 & 0.8065 \\ \hline 
	5040 & 11 & 18.2696 & 0.6021 \\ \hline 
	55440 & 17 & 26.1157 & 0.6510 \\ \hline 
	720720 & 17 & 35.0931 & 0.4844 \\ \hline 
	1441440 & 17 & 37.6072 & 0.4520 \\ \hline 
	4324320 & 19 & 41.6607 & 0.4561 \\ \hline 
	21621600 & 17 & 47.7403 & 0.3561 \\ \hline 
	367567200 & 19 & 58.8074 & 0.3231 \\ \hline 
	6983776800 & 29 & 70.7411 & 0.4099 \\ \hline 
	160626866400 & 29 & 83.8697 & 0.3458 \\ \hline 
	321253732800 & 31 & 86.8251 & 0.3570 \\ \hline 
	9316358251200 & 37 & 101.4324 & 0.3648 \\ \hline 
\end{tabular}
	\caption{Colossally Abundant Integer Primitive Root $q$}           
	\vskip .1 in		
	\label{table-300-C}
\end{table}		
The small sample of primorial integers \autoref{table-300-P} shows that the least primes $p$ with a fixed primorial primitive root $q=2\cdot3\cdots p_k$ is about $p\approx p_k\leq c(p,q)\log q \ll (\log q)(\log\log q)$. Assuming this approximation is correct, the determination of the constant $c(p,q)$ is an interesting problem.

\begin{table}[H]
	\setlength{\tabcolsep}{.35cm}
	\renewcommand{\arraystretch}{1.250092}
	\setlength{\arrayrulewidth}{.975pt}
	\centering
\begin{tabular}{c|c|c|c|c}
	Primitive Root $q$ &$q\bmod 4$ & Least Prime $p$ & $(\log q)(\log\log q)$ & $\frac{p}{(\log q)(\log\log q)}$ \\ \hline 
	$2$ & 2&3 & 0.0000 &0.0000\\ \hline 
	$2\cdot3$ & 2&11 & 1.0450 &10.5268\\ \hline 
	$2\cdot3\cdot5$&2 & 11 & 4.1635 &2.6420 \\ \hline
	$2\cdot3\cdots7$&2 & 13 & 8.9647 &1.4501 \\ \hline
	$2\cdot3\cdots11$&2 & 23 & 15.8544 &1.4507 \\ \hline
	$2\cdot3\cdots13$&2 & 19 & 24.0542 &0.7899 \\ \hline
	$2\cdot3\cdots17$&2 & 41 & 33.8555 &1.2110 \\ \hline
	$2\cdot3\cdots19$&2 & 23 & 44.6922 &0.5146 \\ \hline
	$2\cdot3\cdots23$&2 & 43 & 56.8256 &0.7567 \\ \hline
	$2\cdot3\cdots29$&2 & 41 & 70.4261 &0.5822 \\ \hline
	$2\cdot3\cdots31$&2 & 41 & 84.8143 &0.4834 \\ \hline
	$2\cdot3\cdots37$&2 & 41 & 100.4330 &0.4082 \\ \hline
	$2\cdot3\cdots41$&2 & 43 & 116.9553 &0.3677 \\ \hline
	$2\cdot3\cdots43$&2 & 47 & 134.1117 &0.3505 \\ \hline
	$2\cdot3\cdots47$&2 & 53 & 152.0690 &0.3485 \\ \hline
	$2\cdot3\cdots53$&2 & 79 & 170.9659 &0.4621 \\ \hline
	$2\cdot3\cdots59$&2 & 61 & 190.7386 &0.3198 \\ \hline
	$2\cdot3\cdots61$&2 & 71 & 211.0167 &0.3365 \\ \hline
	$2\cdot3\cdots67$&2 & 73 & 232.0868 &0.3145 \\ \hline
	$2\cdot3\cdots71$&2 & 73 & 253.7626 &0.2877 \\ \hline
	$2\cdot3\cdots73$&2 & 103 & 275.8777 &0.3734 \\ \hline
	$2\cdot3\cdots79$&2 & 83 & 298.6873 &0.2779 \\ \hline
	$2\cdot3\cdots83$&2 & 89 & 322.0312 &0.2764 \\ \hline
	$2\cdot3\cdots89$&2 & 101 & 346.0117 &0.2919 \\ \hline
	$2\cdot3\cdots 97$&2 & 101 & 370.7140 &0.2724 \\ \hline
\end{tabular}
        
\caption{Primorial Integer Primitive Root $q$}           
\vskip .1 in		
\label{table-300-P}
\end{table}		
The small sample of odd primorial integers \autoref{table-300-OP} shows that the least primes $p$ with a fixed primorial primitive root $q=3\cdots p_k$ for $q\equiv 1\bmod 4$ seems to be larger $p> p_k\approx\log q $ but satisfies $p\leq c (\log q)(\log\log q) $ with $c=2$. This is the case with corrected density constant $\delta(q)\ne a_1$.

\begin{table}[H]
	\setlength{\tabcolsep}{.35709cm}
	\renewcommand{\arraystretch}{1.250092}
	\setlength{\arrayrulewidth}{.975pt}
	\centering
\begin{tabular}{c|c|c|c|c}
	Primitive Root $q$ &$q\bmod 4$ & Least Prime $p$ & $(\log q)(\log\log q)$ &  $\frac{p}{(\log q)(\log \log q)}$ \\ \hline 
	$3$ & 3&5 & 1.0986 & 4.5512\\ \hline 
	$3\cdot5$&3 & 13 & 2.7081 & 4.8005\\ \hline
	$3\cdot5\cdot7$&1 & 11 & 4.6540 & 2.3636\\ \hline
	$3\cdot5\cdots11$&3 & 13 & 7.0519 & 1.8435\\ \hline
	$3\cdot5\cdots13$&3 & 23 & 9.6168 & 2.3916\\ \hline
	$3\cdot5\cdots17$&3 & 29 & 12.4500 & 2.3293\\ \hline
	$3\cdot5\cdots19$&1 & 23 & 15.3945 & 1.4940\\ \hline
	$3\cdot5\cdots23$&3 & 31 & 18.5300 & 1.6730\\ \hline
	$3\cdot5\cdots29$&3 & 59 & 21.8972 & 2.6944\\ \hline
	$3\cdot5\cdots31$&1 & 37 & 25.3312 & 1.4606\\ \hline
	$3\cdot5\cdots37$&1 & 47 & 28.9422 & 1.6239\\ \hline
	$3\cdot5\cdots41$&1 & 59 & 32.6557 & 1.8067\\ \hline
	$3\cdot5\cdots43$&3 & 47 & 36.4169 & 1.2906\\ \hline
	$3\cdot5\cdots47$&1 & 59 & 40.2671 & 1.4652\\ \hline
	$3\cdot5\cdots53$&1 & 59 & 44.2374 & 1.3337\\ \hline
	$3\cdot5\cdots59$&3 & 67 & 48.3149 & 1.3867\\ \hline
	$3\cdot5\cdots61$&3 & 71 & 52.4258 & 1.3543\\ \hline
	$3\cdot5\cdots67$&1 & 101 & 56.6305 & 1.7835\\ \hline
	$3\cdot5\cdots71$&3 & 73 & 60.8931 & 1.1988\\ \hline
	$3\cdot5\cdots73$&3 & 83 & 65.1836 & 1.2733\\ \hline
	$3\cdot5\cdots79$&1 & 107 & 69.5531 & 1.5384\\ \hline
	$3\cdot5\cdots83$&3 & 89 & 73.9719 & 1.2032\\ \hline
	$3\cdot5\cdots89$&3 & 97 & 78.4605 & 1.2363\\ \hline
	$3\cdot5\cdots 97$&3 & 113 & 83.0352 & 1.3609\\ \hline
\end{tabular}
        
	\caption{Odd Primorial Integer Primitive Root $q$}           
	\vskip .1 in		
	\label{table-300-OP}
\end{table}		

\section{Open Problems}\label{EXE3535}

\begin{exe}\label{exe3535.0103} {\normalfont Fix an odd primorial $q=3\cdots p_k$ with $q\equiv 1\bmod 4$. Show that the least primes $p$ be with a fixed primitive root $q$ has the lower bound $p> p_k\geq c(p,q)\log q $.
	}
\end{exe}

\vskip .15in 
\begin{exe}\label{exe3535.0105} {\normalfont Show that the least prime $p\geq3$ such that any squarefree integer $q\in \mathscr{S}=\{ q: \mu(q)\ne0\}\subset \Z$ is primitive root mod $p$ is bounded $p\ll 1$.
	}
\end{exe}

\vskip .15in
\begin{exe}\label{exe3535.0107} {\normalfont Find the largest collection $\mathscr{Q}=\{ q\ne\pm1, r^2\}\subset \Z$ of primitive roots such that the least primes $p\ll 1$ is bounded --- in \hyperlink{S3292.333-H}{Section} \ref{S3292.333-H} there is a sample of infinite sequences.
	}
\end{exe}

\vskip .15in
\begin{exe}\label{exe3535.0117} {\normalfont Devise a counterexample, or Prove or disprove that the subset of  primitive roots $\mathscr{E}=\{ q\ne\pm1, r^2\}\subset \Z$ such that the least primes $p\geq3$ are unbounded has zero density, more precisely, 
		\begin{equation}
\delta(\mathscr{E})	=\lim_{x\to\infty}\; \frac{\#\mathscr{E}}{x}=0.\nonumber
		\end{equation}
	}
\end{exe}

\vskip .15in

\begin{exe}\label{exe3535.0123} {\normalfont Develop a divisor-free primitive root test. The standard primitive root test, described in \hyperlink{S3242.500}{Lemma} \ref{lem3242.500}, is totally dependent on the divisors of $p-1$.
}
\end{exe}



\end{document}